\documentclass[10pt]{amsart}
\usepackage{mathrsfs}
\usepackage{}
\usepackage{amsfonts}
\usepackage{amsfonts,latexsym,rawfonts,amsmath,amssymb,amsthm}
\usepackage[plainpages=false]{hyperref}
\usepackage{fancyhdr}
\usepackage{graphicx}
\makeatletter
\renewcommand\@biblabel[1]{}
\makeatother

\numberwithin{equation}{section}
\newcommand{\beq}{\begin{equation}}
\newcommand{\eeq}{\end{equation}}
\newcommand{\beqs}{\begin{eqnarray*}}
\newcommand{\eeqs}{\end{eqnarray*}}
\newcommand{\beqn}{\begin{eqnarray}}
\newcommand{\eeqn}{\end{eqnarray}}
\newcommand{\beqa}{\begin{array}}
\newcommand{\eeqa}{\end{array}}

\def\lra{\longrightarrow}

\def\bc{\begin{center}}
\def\ec{\end{center}}

\def\begeq{\begin{equation}}
\def\endeq{\end{equation}}
\def\and{\quad{\rm and}\quad}

\let\lra=\longrightarrow

\def\mapright\#1{\,\smash{\mathop{\lra}\limits^{\#1}}\,}

\newtheorem{prop}{Proposition}[section]
\newtheorem{theo}[prop]{Theorem}
\newtheorem{lem}[prop]{Lemma}

\newtheorem{cor}[prop]{Corollary}

\begin{document}

\date{}
\author{Gang  ${\rm Tian}^*$ \ \ \ \ \ \ Xiaohua  ${\rm Zhu}^{**}$}

\thanks {* Partially supported by NSFC and NSF Grants}

\thanks {** Partially supported by the NSFC Grants 11271022 and 11331001}

\address{Gang Tian, SMS and BICMR, Peking
University, Beijing, 100871, China, and Department of Mathematics, Princeton
University,  New Jersey,  NJ 02139, USA\\ tian@math.princeton.edu}

\address{Xiaohua Zhu, SMS and BICMR, Peking
University, Beijing 100871, China\\
xhzhu@math.pku.edu.cn}

\title{  Properness of log $F$-functionals}

\begin{abstract}
In this paper, we apply the method developed in [Ti97] and [TZ00] to proving the properness of log $F$-functional on any conic K\"ahler-Einstein manifolds. 
As an application, we give an alternative proof for the openness of the continuity method through conic K\"ahler-Einstein metrics.
\end{abstract}
 \maketitle
	
\tableofcontents

\section *{ 0.  Introduction}
It has been very active to study conic K\"ahler-Einstein metrics in recent years partly because of their use in studying problems in algebraic geometry and
 K\"ahler geometry. For example, they provide a continuity method for establishing the existence of K\"ahler-Einstein metrics on any Fano manifold $M$,
 that is, a compact K\"ahler manifold with positive first Chern class $c_1(M)>0$. Such a continuity is used in the recent solution to Yau-Tian-Donaldson conjecture
 given independently by Tian and Chen-Dondaldson-Sun [Ti12], [CDT13]. The conjecture states that a Fano manifold $M$ admits a K\"ahler-Einstein metric if and only if $M$ is K-stable as defined in [Ti97] and reformulated in [Do02]. The K-stability is closely related to the properness of 
Mabuchi's $K$-energy, or equivalently, the $F$-functional. 
It is proved in [Ti97] that if $M$ admits no non-zero holomorphic vector field, then the existence of K\"ahler-Einstein metrics on $M$ is equivalent 
to the properness of $F$-functional or Mabuchi's $K$-energy. The purpose of this paper is to adapt the arguments in [Ti97] as well as [TZ00]
to show that similar results still hold for conic K\"ahler metrics.

Now let us recall some basics on conic K\"ahler metrics. Let $D$ be a smooth divisor of $M$ with $[D]\in \lambda c_1(M)$ for some  $\lambda>0$ and $S$
be a defining section of $D$. Choose a smooth K\"ahler metric $\omega_0$ with K\"ahler class $[\omega_0]\,=\, 2\pi \,c_1(M)$, then there is 
a Hermitian metric $H_0$ on $[D]$ whose curvature is $\omega_0$. Following computations in [Au84] and [Di88] (also see [Ti87], [DT92]), 
Jeffres-Mazzeo-Rubinstein   introduced
a log $F$-functional on the space of K\"ahler potentials  [JMR11]:
$$\mathcal H(M, \omega_0)=\{ \psi\in C^\infty(M)|~\omega_{\psi}=\omega_0+\sqrt{-1}\partial\bar\partial\psi>0\}.$$ 
This log $F$-functional is an Eular-Langrange energy of conic K\"ahler-Einstein metrics with cone angle $2\pi\beta$ along $D$ and is defined by 
(also see  [LS12])
\begin{align}\label{log-functional} F_{\omega_0,\mu}(\psi)\,=\, J_{\omega_0}(\psi) \,-\,\frac{1}{V}\int_M\,\psi\,\omega^n_0\,-\,\frac{1}{\mu}\log\left(\frac{1}{V}\,
\int_M  \,\frac{1}{|S|_{H_0}^{2\beta}}\,e^{h_0-\mu\psi}\,\omega^n_0\right),
\end{align}
where $\mu\,=\,1-(1-\beta)\lambda\in (0,1)$, $V\,=\,\int_M\,\omega_0^n$ and $h_0$ is the Ricci potential of $\omega_0$ defined by
$${\rm Ric}(\omega_0)\,-\,\omega_0\,=\,\sqrt{-1}\,\partial\bar\partial \,h_0,~~~~\int_M\,\left(e^{h_0}\,-\,1\right)\,\omega_0^n\,=\,0.$$
Note that $J_{\omega_0}(\phi)$ is defined by (see [Au84], [Ti87])
\begin{align}
J_{\omega_0}(\psi)\,=\,\frac{1}{V}\sum_{i=0}^{n-1}\frac{i+1}{n+1}
\int_M\sqrt{-1}\partial\psi\wedge\bar\partial\psi\wedge\omega_0^{i}\wedge\omega_{\psi}^{n-i-1}.\notag
\end{align}

The main result of this paper is the following

\begin{theo}\label{tian-zhu-conic-ke-metric}   
Let $D$ be a smooth divisor of a Fano manifold $M$ with $[D]\in \lambda c_1(M)$ for some $\lambda>0$ such  that there is no non-zero holomorphic field
which is tangent to $D$ along $D$. \footnote{This condition can be removed if $\lambda\ge 1$ by a result in  [Be11], or [SW12].}
Suppose that there exists  a conic  K\"ahler-Einstein metric on $M$ with  cone angle $2\pi\beta\in (0,2\pi)$ along $D$. 
Then there exists two uniform constants $\delta$  and $C$ such that
\begin{align}\label{proper-inequality}F_{\omega_0,\mu}(\psi)\,\ge\, \delta \,I_{\omega_0}(\psi)\, -\,C,~~~\forall~\psi \in \mathcal H(M,\omega_0),
\end{align}
where $I_{\omega_0}(\psi)\,=\,\frac{1}{V}\int_{M}\,\phi\,(\omega_0^n\,-\,\omega_{\psi}^n)$.
\end{theo}

Combined with a result in [JMR11], Theorem \ref{tian-zhu-conic-ke-metric} implies that there exists a  conic   K\"ahler-Einstein metric on $M$  along
$D$ with cone angle $2\pi\beta\in (0,2\pi)$ if and only if $F_{\omega_0,\mu}(\cdot)$ is proper. This generalizes Tian's theorem in [Ti97] in the case of
smooth K\"ahler-Einstein manifolds. 

As an application of Theorem \ref{tian-zhu-conic-ke-metric} or more precisely, its weaker version Theorem \ref{tian-zhu-conic-ke-metric-weak} in Section 6,  
we give an alternative proof for the openness of the continuity method through conic K\"ahler-Einstein metrics. The proof of such an openness was first 
sketched by Donaldson [Do11] as an application of the $C^{2,\alpha;\beta}$ Schauder estimates Donaldson developed for conic K\"ahler metrics. 
Since the space $C^{2,\alpha;\beta}$ will depend on the cone angles $2\pi\beta$ of metrics, the usual Implicit Function Theorem could not be applied directly to prove the openness. Instead, Donaldson consider a family of linear elliptic operators associated to approximated conic metrics to get a prior Schauder estimates needed
for proving the openness. Our proof is to use the perturbation method first introduced in [Ti12] to approximate conic K\"ahler-Einstein metrics
by smooth K\"ahler metrics, then we apply the Implicit Function Theorem to approximated smooth K\"ahler metrics and take limit (cf. Section 7).  
To assume the limit exists, we need to establish a prior $C^0$ and $C^2$-estimates for the K\"ahler potentials associated to those
approximated metrics. With these a prior estimates, we can take the limit to get a weak conic K\"ahler-Einstein metric. This metric 
is in fact in sense of $C^{2,\alpha;\beta}$ Schauder theory by the regularity theorem in [JMR11].

The proof of Theorem \ref{tian-zhu-conic-ke-metric} is an adaption of that in [Ti97] for smooth K\"ahler-Einstein manifolds.  
In our situation, there are some technical issues we need to make sure. First we need to show how to smooth singular metrics near the conic  
K\"ahler-Einstein metric. We will use a family
of twisted K\"ahler-Ricci flows with initial values given by smooth metrics which approximate conic metrics (see Section 5, 6). 
Then we shall deal with  the local  smoothing behavior of these flows as well as the local convergence of flows when the initial values vary. 
Note that as a parabolic version of twisted K\"ahler-Einstein metric equation, which was first introduced by Song-Tian [ST12],   
the twisted K\"ahler-Ricci flow has been also studied by many people, such as Collins-Szekelyhihi [CS12], Liu-Zhang [LZ14], Liangmin Shen [Sh14] etc..

The organization of this paper is as follows: In Section 1, we recall some basics on conic K\"ahler metrics. 
In Section 2, we prove the lower bound of log $F$-functional $F_{\omega_0,\mu}(\cdot)$. In Section 3,  we introduce a
family of smooth K\"ahler metrics to approximate the conic K\"ahler metrics discussed in Section 2. 
In Section 4, we introduce a family of twisted  K\"ahler-Ricci flows to smooth the approximated metrics in Section 3, then in Section 5, we prove the
local convergence of these flows. Theorem \ref{tian-zhu-conic-ke-metric} will be proved in Section 6.   
In Section 7, we apply Theorem \ref{tian-zhu-conic-ke-metric} to give a proof of the openness for the continuity method through conic K\"ahler-Einstein metrics
which was first given by Donaldson.

\section{Conic K\"ahler metrics}

Let $S$ be a defining function of $D$ and $H_0$ a Hermitian metric on $D$ induced by $\omega_0$. Then  it is easy to see that  $|S|^{2\beta} =|S|_{H_0}^{2\beta}\in C^{2,\alpha;\beta} (M)$  for any $\alpha\in (0,1)$
  in sense of [Do11].

Moreover, one can check that
$\omega^*=\omega_0+\delta\sqrt{-1}\partial\bar{\partial} |S|^{2\beta}$ is a conic K\"ahler metric with cone angle $2\pi\beta$ along $D$,  as long as  the number $\delta$ is  sufficiently small  (cf. [Br11]).
There is an important property of $\omega^*$  shown in [JMR11]
 that the  bisectional holomorphic curvature of  $\omega^*$ is uniformly bounded from above on $M\setminus D$.

Let  $h^*$ be a log Ricci potential of $\omega^*$ defined by
\begin{align}\label{ricci-potential-omega-star}
\sqrt{-1}\partial\bar{\partial}h^*={\rm Ric}(\omega^*)-\mu\omega^*-2\pi(1-\beta)[D].
\end{align}
Then  we have
\begin{align}\label{ricci-potential-relation}
h^*= h_0-\mu\delta |S|^{2\beta}-\log\frac{|S|^{2(1-\beta)}(\omega^*)^n}{\omega_0^n}+const,
\end{align}
where $h_0$ is a Ricci potential of $\omega_0$.  A direct computation shows that  $h^*\in C^{,\gamma;\beta}(M)$, where $\gamma=\min(\frac{2}{\beta}-2,1).$

In general,   a K\"ahler potential of conic K\"ahler metric is not necessary  in $C^{2,\alpha;\beta} (M)$.  But, for a  conic K\"ahler-Einstein metric
 $\omega_{CKE}=\omega_\phi=\omega_0+\sqrt{-1}\partial\bar{\partial}\phi$ with angle $2\pi\beta$ along $D$, $\phi$ should be in   $C^{2,\alpha_0;\beta} (M)$
 for some positive number  $\alpha_0 \le\gamma.$  This is because
$\omega_{CKE}$ satisfies  a  conic K\"ahler-Einstein  metric  equation,
\begin{align}\label{conic-KE-equation}
{\rm Ric}(\omega)-\mu\omega-2\pi(1-\beta)[D]=0.
\end{align}
Then   $\phi$ satisfies  a  non-degenerate  complex Monge-Amp\`ere equation,
 \begin{align}\label{complex-MA-equation-conic}
(\omega^*+\sqrt{-1}\partial\bar\partial\phi)^n=e^{h^*-\mu\phi}(\omega^*)^n.
\end{align}
The  $C^{2,\alpha;\beta}$-regularity theorem established in [JMR11] (also in [GP13]) implies that $ \phi\in C^{2,\alpha_0;\beta} (M)$ for some $\alpha_0\le \gamma$.

For any  positive number $\alpha\le \alpha_0$,  we introduce a  space of  $C^{2,\alpha; \beta}$ K\"ahler potentials by
  \begin{align}&\mathscr{H}^{2,\alpha; \beta}(M,\omega_0)\notag\\
&=\{\psi \in C^{2,\alpha; \beta}(M)|~  \omega_\psi=\omega_0+\sqrt{-1}\partial\bar\partial\psi ~{\rm ~is ~  a ~conic~ Kaehler~  metric~ on} ~M\}.
\notag\end{align}
One can show that both functionals $F_{\omega_0,\mu}(\cdot)$ and $I_{\omega_0}(\cdot)$ are well-defined on $\mathscr{H}^{2,\alpha;\beta}$~~
\newline $(M,\omega_0)$.

\begin{lem}\label{functional-smoothing} For any $\psi\in\mathcal H(M,\omega_0) $, there a sequence of $\psi_\delta\in \mathscr{H}^{2,\alpha;\beta}(M,\omega_0)$ such that
$$ F_{\omega_0,\mu}(\psi)=\lim_{\delta\to 0}  F_{\omega_0,\mu}(\psi_\delta)$$
and
$$I_{\omega_0}(\psi)=\lim_{\delta\to 0} I_{\omega_0}(\psi_\delta).$$
\end{lem}

\begin{proof} In fact, one can choose $\psi_\delta=\psi+\delta |S|^{2\beta}$ to verify the lemma.
\end{proof}.

\section{Lower bound of  $F_{\omega_0,\mu}(\cdot)$}

In this section, we  use  the continuity method of Ding-Tian  in  [DT92] ( also see  [Ti97]) to  study   the lower bound of  log $F$-functional  $F_{\omega_0,\mu}(\cdot)$.   This method will be extended to prove
Main Theorem \ref{tian-zhu-conic-ke-metric}  in this paper.
It is worthy to mention  that   the lower bound of  $F_{\omega_0,\mu}(\cdot)$  can be   obtained by using  a general theorem of  Berndtsson for the uniqueness problem of
special K\"ahler potetials  in [Be11].  Berndtsson's method is based on applications of  geodesic theory  about  K\"ahler potentials space studied in  [Se92], [Do98], [Ch00], etc..

\begin{theo}\label{Bern_thm}
Let  $\omega=\omega_{CKE}=\omega_\phi$  be a conic K\"ahler-Einstein metric  on $M$ with cone angle $2\pi\beta$ along $D$. Then $\phi$ obtains the minimum
 of $F_{\omega_0,\mu}(\cdot)$ on $\mathscr{H}^{2,\alpha, \beta}(M,\omega)$.  In particular,
\begin{equation}
F_{\omega_0,\mu}(\psi)\ge-c(\omega_0,\mu),~\forall~ \psi\in\mathscr{H}^{2,\alpha;\beta}(M,\omega_0).
\end{equation}
\end{theo}

\begin{proof}
For any $\psi\in \mathscr{H}^{2,\alpha; \beta}(M,\omega_0)$,  log Ricci potential of  $\hat\omega=\omega_{\psi}$ is given by
 $$h_{\hat\omega}=-\log\frac{\omega_\psi^n}{(\omega^*)^n} -\mu\psi+h^*+const.$$
 Then $h_{\hat\omega}\in  C^{,\alpha;\beta}(M)$.  We consider the following  complex Monge-Amp\`ere  equations with a parameter $t\in [0,\mu]$:
\begin{equation}\label{backward_MAE}
({\hat\omega}+\sqrt{-1}\partial\bar\partial\varphi)^n=e^{h_{\hat\omega}-t\varphi}{\hat\omega}^n.
\end{equation}
  By the assumption,  there exists  a solution $\varphi_\mu=\phi-\psi+const$ at $t=\mu=1-(1-\beta)\lambda$.   Note that  the kernel of operator $(\Delta_{\omega}+\mu)$ is zero (cf. [Do11]).
Then  by the Donaldson's linear theory for Laplace operators associated to conic  metrics, we can apply  Implicity   Function Theorem to
show that there exists  a $\delta>0$ such that  (\ref{backward_MAE}) is solvable in the space  $\mathscr{H}^{2,\alpha;\beta}(M,\omega_0)$ on  any $t\in (\mu-\delta, \mu]$.

Set
$$E=\{s\in [0,\mu]|~  (\ref{backward_MAE}) \mbox{ is solvable on }~t=s ~{\rm in } ~ \mathscr{H}^{2,\alpha'; \beta}(M,\omega_0) ~{\rm for ~some}~\alpha'\le \alpha \}.$$
We  want  to prove $E=[0,\mu].$
Clearly,  $E$ is non-empty since $(\mu-\delta, \mu]\subset E$.   On the other hand,  it is easy to see that   (\ref{backward_MAE}) are equivalent to Ricci curvature equations,
\begin{align}\label{log-twisted-equation-t}{\rm Ric}(\hat \omega_{\varphi})=t\hat\omega_{\varphi}+ (\mu-t)\hat\omega+ 2\pi (1-\beta)[D], ~t~\in [0,\mu].
\end{align}
Then
 \begin{align}\label{ricci-equation-t}{\rm Ric}(\hat\omega_{\varphi_t})>t\hat\omega_{\varphi_t}, ~{\rm in}~ M\setminus D.
\end{align}
Thus  the first non-eigenvalue of $\Delta_t$ is strictly  bigger than $t$ [JMR11],
where  ${\Delta}_t$ is the Laplace operator associated to $\omega_t$  and ${\omega}_t=\hat\omega_{\varphi_t}=\hat{\omega}+\sqrt{-1}\partial\bar\partial\varphi_t$.
 It follows that  the kernel of operator $(\Delta_{t}+t)$ is zero on  any $t\in E$.   By  Implicity   Function Theorem,  $E$ is  an  open set.
 It  remains  to prove that   $E$ is  also a closed set.  This is related to apriori estimates for solution $\varphi_t$ of   (\ref{backward_MAE}) on  $t\in E$ below.

First we deal with the  $C^0$-estimate.  We may assume that $t\ge \delta$   by the implicity theorem   since  (\ref{backward_MAE}) is solvable  at  $t=0$ [JMR11].
By a direct computation,  we have
\begin{align}\nonumber
 \frac{d}{dt}(I_{\hat\omega}(\varphi_t)-J_{\hat{\omega}}(\varphi_t))&=-\frac{1}{V}\int_M\varphi_t{\Delta}_t\dot{\varphi_t}{\omega}_t^n\\ \nonumber
 &=\frac{1}{V}\int_M({\Delta}_t\dot{\varphi_t}+t\dot{\varphi_t}){\Delta}_t\dot{\varphi_t}{\omega}_t^n.\nonumber
\end{align}
  Note that
$${\Delta}_t\dot{\varphi_t}=-t\dot{\varphi_t}-\varphi_t$$
 by differentiating $\int_M e^{h_{\hat{\omega}}-t\varphi}\hat{\omega}^n=V$.  By the fact that  the first non-eigenvalue of $\Delta_t$ is strictly  bigger than $t$, we get
$$\frac{d}{dt}(I_{\hat{\omega}}(\varphi_t)-J_{\hat{\omega}}(\varphi_t))\ge 0.$$
 This means that   $I_{\hat{\omega}}(\varphi_t)-J_{\hat{\omega}}(\varphi_t)$ is  increasing in $t$. Thus
 $$I_{\hat{\omega}}(\varphi_t)\le (n+1)I_{\hat{\omega}}(\varphi_\mu)\le C.$$
By using the Green formula [JMR11],  we derive
$${\rm osc}(\varphi_t)\le C.$$

To get the $C^2$-estimate, we rewrite (\ref{backward_MAE}) as
\begin{equation}\label{backward_MAE-2}
(\omega^*+\sqrt{-1}\partial\bar\partial\varphi^*)^n=e^{h^*-t\varphi^*-(\mu-t)(\psi-\delta |S|^{2\beta})  } (\omega^*)^n,
\end{equation}
where $\varphi^*=\varphi-\delta |S|^{2\beta}+\psi$ and $h^*$ is the log Ricci potential of $\omega^*$ as in (\ref{ricci-potential-omega-star}).
Since ${\rm Ric}(\omega_{\varphi})>0$,  by  the Chern-Lu inequality [Cher68], [Lu68],   we have
\begin{align}\label{Chern-Lu inequality-1}
\Delta_t \log{\rm  tr}_{  \omega_{t}}(\omega^*)\ge -a(\omega^*)\, {\rm  tr}_{  \omega_{t}}(\omega^*), ~{\rm in}~M\setminus D,
\end{align}
where $a=a(\omega^*)$ is a uniform constant  which depends only on the upper bound of bisectional holomorphic curvature of  $\omega^*$,  and so  it
depends only on  $\omega_0$ and the divisor $D$.  Set
$$u=\log{\rm  tr}_{  \omega_{t}}(\omega^*) -(a+1)\varphi^*.$$
Then  there exists a uniform constant  $C=C(\sup_M  \phi, \sup_M  \psi)$ such that
$$\Delta_t  u\ge e^{u-C(a+1)}-n(a+1).$$
By the maximum principle as in [JMR11],  it follows
\begin{align}\label{upper-bound-phi-1}\omega_t\ge C^{-1}\omega^*.
\end{align}
By (\ref{backward_MAE-2}), we also get
\begin{align}\label{lower-bound-phi-1}\omega_t\le C \omega^*.
\end{align}

Once (\ref{upper-bound-phi-1}) and (\ref{lower-bound-phi-1}) hold,   we can apply  the $C^{2,\alpha;\beta}$-regularity theorem  in  [JMR11] to  show
$$\|\varphi\|_{C^{2,\alpha';\beta}(M)}\le C$$
 for some $\alpha'\le \alpha$. Thus   $\varphi_t\in  \mathscr{H}^{2,\alpha'; \beta}(M,\omega_0)$.  This implies that $E$ is a closed set and so $E=[0,\mu]$.

 By a direct computation as in [DT92], we have
 \begin{align}\label{integral-dt} t\left(J_{\hat{\omega}}(\varphi_\mu)-\frac{1}{V}\int_M\varphi_\mu\hat{\omega}^n\right)=-\int_0^t  (I-J)_{\hat{\omega}}(\varphi_s)ds\le 0, ~\forall~t\le \mu.
\end{align}
Note that  $\int_M e^{h_{\hat{\omega}}-\mu\varphi_\mu}\hat{\omega}^n=V$.  Thus
$$F_{\hat{\omega},\mu}(\varphi_\mu)=J_{\hat{\omega}}(\varphi_\mu)-\frac{1}{V}\int_M\varphi_\mu\hat{\omega}^n\le 0.$$
By  the cocycle condition  of  log  $F$-functional,  it follows
\begin{align}
F_{\omega,\mu}({\psi-\phi})=-F_{\hat{\omega},\mu}(\varphi_\mu)\ge 0.\notag
\end{align}
Again by  the cocycle condition,   we get
\begin{align}\label{cocycle-condition}
F_{\omega_0,\mu}(\psi)=F_{\omega,\mu}(\psi-\phi)+F_{\omega_0,\mu}(\phi)\ge F_{\omega_0,\mu}(\phi).
\end{align}
 Hence we prove that  $F_{\omega_0,\mu}(\cdot)$ takes the minimum at  $\phi$.
\end{proof}

\begin{cor}\label{Bern_thm-coro}
Suppose that there exists   a conic K\"ahler-Einstein metric  on $M$ with  cone angle $2\pi\beta$ along $D$.  Then
\begin{align}
F_{\omega_0,\mu}(\psi)\ge-c(\omega_0,\mu),~\forall~ \psi\in\cup_{0<\alpha'\le \alpha_0}\mathscr{H}^{2,\alpha';\beta}(M,\omega_0).\notag
\end{align}
\end{cor}

\section{Approximation of conic K\"ahler  metrics}

In this section, we construct approximated smooth K\"ahler potentials  of solution $\varphi_t$ of (\ref{backward_MAE}) on  each $t\in (0,\mu)$ by solving certian complex Monge-Amp\`ere equation.
 First we shall smooth the conic metric
 $\hat \omega=\omega_{\psi}$.  Note
$$(\hat\omega)^n=f_0\omega_0^n, $$
where $f_0=g\frac{1}{|S|^{2-2\beta}}$  for some  $L^\infty$-function $g$.   In particular,   $f_0$ is a $L^p$-function.

Take a family of smooth functions $f_\delta$ with $\int_M f_\delta\omega_0^n=\int_M \omega_0^n$ such that $f_\delta$ converge to $f_0$ in  $L^p$ as $\delta\to 0$.
Then by  the Yau's solution to Calabi's  problem,
there are a family of  K\"ahler potentials $\Psi_\delta$, which  solve equations $(\delta>0)$,
$$(\omega_0+\sqrt{-1}\partial\bar\partial\Psi_{\delta})^n=f_\delta\omega_0^n.$$
By  the Kolodziej's  H\"older estimate  [Kol08], $\Psi_{\delta}$  converge to $\psi$ in  the $C^\alpha$-norm modulo constants as $\delta\to 0.$
For simplicity, we set  $\omega_\delta=\omega_0+\sqrt{-1}\partial\bar\partial\Psi_{\delta}.$

We  modify  (\ref{log-twisted-equation-t}) by a family of Ricci curvature equations with parameter $\delta\in (0,\delta_0]$ for each  $t\in [0,\mu]$,
\begin{align}\label{twisted-equation-t}{\rm Ric}(\omega^\delta_{\varphi_{\delta}})=t\omega^\delta_{\varphi_{\delta}}+ (\mu-t)\omega_\delta+ (1-\beta)\eta_\delta,
\end{align}
where $\omega^\delta_{\varphi_{\delta}}=\omega_\delta+\sqrt{-1}\partial\bar\partial\varphi_{\delta}$ and $\eta_\delta=\lambda\omega_0+\sqrt{-1}\partial\bar\partial \log(\delta+|S|^2).$
(\ref{twisted-equation-t}) are in fact a family of twisted K\"ahler-Einstein metric eqautions  associated to  positive $(1,1)$-forms $\Omega= (\mu-t)\omega_\delta+ (1-\beta)\eta_\delta$ [ST12].
One can check that (\ref{twisted-equation-t}) are equivalent to the following  complex Monge-Amp\`ere equations,
\begin{equation}\label{smooth-continuity-modified}
(\omega_\delta+\sqrt{-1}\partial\bar\partial\varphi_\delta)^n=e^{h_\delta-t\varphi_\delta}\omega_\delta^n,
\end{equation}
where $h_\delta$ are  twisted Ricci potentials of $\omega_\delta$ defined by
\begin{align}\label{twisted-h-delta}\sqrt{-1}\partial\bar\partial h_\delta= {\rm Ric}(\omega_\delta)-\mu\omega_\delta- (1-\beta)\eta_\delta.
\end{align}
We shall study the solutions of  (\ref{smooth-continuity-modified})  and their convergence as $\delta\to 0.$

Rewrite   (\ref{twisted-equation-t})  as
\begin{align}\label{twisted-equation-t-2}{\rm Ric}(\omega^\delta_{\varphi_{\delta}})=t\omega^\delta_{\varphi_{\delta}}+(\mu-t)\omega_0+ (1-\beta)\eta_\delta  +(\mu-t)\sqrt{-1}\partial\bar\partial  \Psi_{\delta}.
\end{align}
 Then  equations (\ref{smooth-continuity-modified})  are equivalent to
\begin{equation}\label{twisted-MA-equation}
(\omega_0+\sqrt{-1}\partial\bar\partial\hat\varphi_\delta)^n=\frac{1}{(\delta+|S|^2)^{1-\beta}}e^{h_0-t\hat\varphi_\delta -(\mu-t)\Psi_\delta}\omega_0^n,~~t\in[0,\mu],
\end{equation}
where  $\hat \varphi_{\delta}= \hat \varphi_{t, \delta}=\varphi_{t, \delta} + \Psi_{\delta}$.

As in [Ti12],  for a fixed $\delta>0$,  we define a family of twisted F-functionals with parameter  $t\in (0,\mu ]$ as follows,
 \begin{align}\label{twisted-f-functional-t}
F_{t, \delta}(\varphi)=J_{\omega_0}(\varphi)-\frac{1}{V}\int_M\varphi\omega^n_0-\frac{1}{t}\log
\left(\frac{1}{V}\int_Me^{\hat h_\delta-t\varphi}\omega_0^n\right),
\end{align}
where
\begin{align}
\hat h_\delta=h_0-(1-\beta)\log (\delta+|S|_0^2)+(t-\mu)\Psi_{\delta}+C_\delta,~~\int_M(e^{\hat h_\delta}-1)\omega^n_0=0.\notag
\end{align}
Then all $F_{t, \delta}(\cdot)$ are proper for  any  $t\in (0,\mu ), \delta\in (0,\delta_0]$ since  log $F$-functionals $F_{\omega_0, t}(\cdot)$ defined in (\ref{log-functional})  are  proper   for  any  $t\in (0,\mu )$. The latter follows from a result in  [LS12]  by  using the fact that
$F_{\omega_0, \mu }(\cdot)$  is bounded from below according to Theorem  \ref{Bern_thm}.  By the Green formula,  we get
  $${\rm osc}_{M}\hat \varphi_{t,\delta} \le C(I_{\omega_0}(\hat \varphi_{t,\delta})+1) \le C',~\forall ~t\in (0,\mu ), $$
where the constants $C,C'$ depend only on $t$.   Note that  all higher order estimates
for solutions  $\hat \varphi_{t,\delta}$ depend only on $\delta$ and  their $C^0$-norm.  Thus  by using the continuity method as in the proof of Theorem  2.5 in [Ti12],
(\ref{twisted-MA-equation}), and so (\ref{smooth-continuity-modified})  are solvable on  any  $t\in (0,\mu), \delta\in (0,\delta_0]$.

Next we improve   higher order estimates for solutions  $\hat \varphi_{t,\delta}$  to show  that they are independent of $\delta>0$.   Let's  introduce a family of smooth  K\"ahler potentials $\Phi^\beta_\delta$ ($\delta >0)$ constructed by
Guenancia-Paun   in [GP13].    Such  $\Phi^\beta_\delta$ have property:

1) $\Phi^\beta_\delta$ converge to   $\Phi^\beta_0=k|S|^{2\beta}$  as $\delta\to 0$ in  sense of H\"older-norm.

2) Let
$$h_{\kappa_\delta}=-\log\frac{\kappa_\delta^n}{\omega_0^n} -\Phi^\beta_\delta+h_0$$
   be  Ricci potentials of $\kappa_\delta=\omega_0+ \sqrt{-1}\partial\bar\partial \Phi^\beta_\delta$.  Then
$\frac{1}{\delta+|S|^{2\beta}} e^{h_\kappa}$ is uniformly bounded on $\delta$.

3) The  bisectional holomorphic curvatures  $R_{\delta i\bar ij\bar j}$ of  $\kappa_\delta$ satisfy: for any K\"ahler metric
$\omega_{\phi+ \Phi^\beta_\delta}=\kappa_\delta+\sqrt{-1}
\partial\bar\partial\phi$, it holds
\begin{align}\label{gp-inequality} &\sum_{i<j} (\frac{1+\phi_{ i\bar i}} {1+\phi_{ j\bar j}} +\frac{1+\phi_{ j\bar j}} {1+\phi_{ i\bar i}}  -2)R_{\delta i\bar ij\bar j} -C_0{\rm tr}_{\kappa_\delta}(\omega_{\phi+ \Phi^\beta_\delta})
 \Delta_{ \omega_{\phi+ \Phi^\beta_\delta}} \Phi^\beta_\delta\notag\\
&+ \Delta_{ \kappa_{\delta}}  \log(\frac{\kappa_\delta^n}{\omega_0^n}\times (\delta+|S|^2)^{1-\beta})\notag\\
&\le C \sum_{i<j} (\frac{1+\phi_{ i\bar i}} {1+\phi_{ j\bar j}} +\frac{1+\phi_{j\bar j}}{1+\bar\phi_{ i\bar i}}) +C {\rm tr}_{\kappa_\delta}(\omega_{\phi+ \Phi^\beta_\delta })\times {\rm tr}_{\omega_{\phi+ \Phi^\beta_\delta}}(\kappa_\delta)+C,
\end{align}
where $C_0$ and $C$ are two uniform constants.

The following is about uniform aprior $C^2$-estimate for $\varphi= \varphi_{t,\delta}$.

\begin{lem}\label{c2-estimate-twisted-metrics-lemma} For any  $t\in (0,\mu),\delta\in(0,\delta_0]$,  it holds
\begin{align}\label{c2-estimate-twisted-metrics}
C^{-1}  \kappa_\delta\le \omega^\delta_{\varphi_{t,\delta}} \le C \kappa_\delta.
\end{align}
Here $C$ is a uniform constant  which depends only on the metric  $\hat\omega$ and $t$.

\end{lem}

\begin{proof} Let $\bar\varphi=\bar\varphi_\delta=\hat\varphi_\delta-\Phi^{\beta}_\delta$. Then (\ref{twisted-MA-equation}) are  equivalent to
\begin{equation}\label{twisted-MA-equation-3}
(\kappa_\delta+\sqrt{-1}\partial\bar\partial\bar\varphi)^n=\frac{1}{(\delta+|S|_0^2)^{1-\beta}} e^{h_{\kappa_\delta}-t\bar\varphi -(\mu-t)\Psi_{\delta} -(1-t)\Phi^\beta_\delta }\kappa_\delta^n,~~t\in (0,\mu_0).
\end{equation}
Following Yau's  $C^2$-estimate in [Yau78], we have
\begin{align}\label{c2-yau} &  -\Delta_{\omega_{\varphi_\delta}^\delta} \log {\rm tr}_{\kappa_\delta} (  \omega^\delta_{\varphi_{\delta}} )\notag\\
&\le\frac{1}{{\rm tr}_{\kappa_\delta}  (  \omega^\delta_{\varphi_{\delta}} ) }\sum_{i<j} (\frac{1+\bar\varphi_{ i\bar i}} {1+\bar\varphi_{ j\bar j}} +\frac{1+\bar\varphi_{ j\bar j}}{1+\bar\varphi_{ i\bar i}} -2)R_{\delta i\bar ij\bar j} \notag\\
&+\frac{1}{{\rm tr}_{\kappa_\delta}  (  \omega^\delta_{\varphi_{\delta}} )} \Delta_{ \kappa_{\delta}}( t\bar\varphi +(\mu-t)\Psi_{\delta} +(1-t)\Phi^\beta_\delta - \bar h_{\kappa_\delta} ).
\end{align}
  On the other hand,  by
$$B\kappa_\delta + \sqrt{-1}\partial\bar\partial \Psi_{\delta}\ge 0,$$
it is easy to see
\begin{align}\Delta_{   \omega^\delta_{\varphi_{\delta}} } \Psi_{\delta}\ge \frac{  \Delta_{ \kappa_{\delta}} \Psi_{\delta}  }  {{\rm tr}_{\kappa_\delta}    (  \omega^\delta_{\varphi_{\delta}} )}
-nB {\rm tr}_{ \omega^\delta_{\varphi_{\delta}}}(\kappa_\delta).\notag
\end{align}
 Using the fact
\begin{align}\label{laplace-h}\Delta_{ \kappa_{\delta}} h_0 \ge  -A,
\end{align}
we get
\begin{align} &\frac{1}{{\rm tr}_{\kappa_\delta} (   \omega^\delta_{\varphi_{\delta}} )} \Delta_{ \kappa_{\delta}}( t\bar\varphi +(\mu-t)\Psi_{\delta} +(1-t)\Phi^\beta_\delta  - \bar h_{\kappa_\delta} ) -\Delta_{ \omega^\delta_{\varphi_{\delta}}  } \Psi_{\delta} \notag\\
&\le \frac{1}{{\rm tr}_{\kappa_\delta} (    \omega^\delta_{\varphi_{\delta}})} \Delta_{ \kappa_{\delta}}  \log(\frac{\kappa_\delta^n}{\omega_0^n}\times (\delta+|S|^2_0)^{1-\beta})
+\frac{n(1-t) +A}{{\rm tr}_{\kappa_\delta} (    \omega^\delta_{\varphi_{\delta}})}
+t +nB{\rm tr}_{   \omega^\delta_{\varphi_{\delta}}}(\kappa_\delta).
\notag
\end{align}
Thus   by  the Guenancia-Paun inequality (\ref{gp-inequality}) for metrics   $\omega^\delta_{\varphi_{\delta}}$,  we deduce from (\ref{c2-yau}),
\begin{align}& - \Delta_{    \omega^\delta_{\varphi_{\delta}}} (\log {\rm tr}_{\kappa_\delta} (  \omega^\delta_{\varphi_{\delta}}) +C_0 \Phi^\beta_\delta - \Psi_{\delta} ) \notag\\
&\le C'  {\rm tr}_{   \omega^\delta_{\varphi_{\delta}}}(\kappa_\delta)+C'.\notag
\end{align}

Let
$$u=\log{\rm  tr}_{\kappa_\delta}(    \omega^\delta_{\varphi_{\delta}})+C_0 \Phi^\beta_\delta - \Psi_{\delta} -(C'+1)\bar\varphi_\delta.$$
Then
$$\Delta_{    \omega^\delta_{\varphi_{\delta}}} u\ge {\rm  tr}_{   \omega^\delta_{\varphi_{\delta}}} (\kappa_\delta)-C''.$$
By the maximum principle,  it follows
\begin{align}\label{upper-bound-phi-delta}    \omega^\delta_{\varphi_{\delta}}=\kappa_{\delta} +\sqrt{-1}\partial\bar\partial\bar\varphi_\delta\ge C^{-1}\kappa_\delta.
\end{align}
By (\ref{twisted-MA-equation-3}), we can also get
 \begin{align}\label{lower-bound-phi-delta}      \omega^\delta_{\varphi_{\delta}}\le C\kappa_\delta.
\end{align}

\end{proof}

\begin{theo}\label{phi-delta-convergence} For any  $t\in (0,\mu)$,  it holds
\begin{align}
\lim_\delta  \varphi_{t, \delta}=\varphi_t\notag
\end{align}
in sense of H\"older-norm.
\end{theo}

\begin{proof} First we claim that $\varphi_{t, \delta}$ converges  to a  $C^{2,\alpha;\beta}$-solution  of (\ref{backward_MAE}) as $\delta\to 0$.
In fact, by  the Kolodziej's H\"older estimate,  we see that  $\hat \varphi_{t,\delta}$ converge to a H\"older continuous  solution  $\phi'$ of following complex Monge-Amp\`ere equation in  the current sense,
\begin{equation}\label{smooth-continuity-2}
(\omega_0+\sqrt{-1}\partial\bar\partial\phi')^n= \frac{1}{|S|^{2-2\beta}}e^{h_0- t \phi' +( t-\mu)\psi}\omega_0^n.
\end{equation}
Clearly,  (\ref{smooth-continuity-2}) is nothing, just  (\ref{backward_MAE}).  Since $\omega^*=\omega_0+\sqrt{-1}\partial\bar\partial\Phi_0^\beta$ is
equivalent to $\hat\omega$,  by Lemma \ref{c2-estimate-twisted-metrics-lemma},  we get
\begin{align}\label{c2-limit-phi-delta-t}
C^{-1}  \hat\omega\le \omega_{\phi'} \le C \hat\omega, ~{\rm in}~ M\setminus D,
\end{align}
where $C$ is a uniform constant.   Note that  (\ref{smooth-continuity-2}) implies that $\phi'- \psi$  is a solution of  (\ref{backward_MAE}).  Thus by the $C^{2,\alpha;\beta}$  reguarity  theorem,
 $\phi'- \psi$ is a $C^{2,\alpha;\beta}$-solution of   (\ref{backward_MAE}).   This proves the claim.

On the other hand,  according to  the proof in Theorem \ref{Bern_thm}, it is easy to see that $C^{2,\alpha;\beta}$-solution of  (\ref{backward_MAE})
 as a  twisted K\"ahler-Einstein metric  is unique. Thus  $\phi'- \psi$ must be  $\varphi_{ t}$. The theorem is proved.

\end{proof}

\section{Smoothing of twisted Ricci potentials}

Define a Log Ricci potential $h_t$  of  $\hat\omega_{\varphi_t}$ of solution of (\ref{backward_MAE}) on  $t$ by
$$\sqrt{-1}\partial\bar\partial h_t={\rm Ric}(\hat\omega_{\varphi_{t}}) - \mu\hat\omega_{\varphi_t}- 2\pi (1-\beta)[D], ~\int_M  e^{ h_t} \hat\omega_{\varphi_{t}}^n=V, $$
and  define a twisted  Ricci potential of  $   \omega^\delta_{\varphi_{t, \delta}}= \omega_\delta+\sqrt{-1}\partial\bar\partial  \varphi_{t,\delta} $ of solution of (\ref{smooth-continuity-modified}) on  $t$ by
$$\sqrt{-1}\partial\bar\partial h_{t,\delta}=  {\rm Ric}(\omega_{\varphi_{t, \delta}}^\delta) -\mu\omega_{\varphi_{t, \delta}}^\delta-(1-\beta)\eta_\delta, ~\int_M  e^{ h_{t,\delta}}  ( \omega^\delta_{\varphi_{t,\delta}})^n=V.$$
 Then it is easy to see
$$  h_t=-(\mu-t)\varphi_t+ const, $$
and
 $$ h_{t, \delta}=-(\mu-t)\varphi_{t, \delta} + const.$$
Thus by Theorem \ref{phi-delta-convergence}, we have

\begin{lem}For any  $t\in (0,\mu)$,  it holds
\begin{align}
\lim_{\delta\to 0}  h_{t,\delta}=h_t\notag
\end{align}
in sense of H\"older-norm.
\end{lem}

To smooth $ h_{t,\delta}$ for each  fixed $\delta\in (0,\delta_0)$,  we introduce the following  twisted K\"ahler-Ricci flow,
 \begin{align}\label{twisted-K-R-flow}
&\frac{\partial}{\partial s} \omega_{\psi}^\delta= -{\rm Ric} (\omega_{\psi}^\delta ) +\mu\omega_{\psi}^\delta+(1-\beta)\eta_\delta,\notag\\
& \omega_{\psi}^\delta(\cdot, 0)=\omega^\delta_{\psi_{0, \delta}}=\omega^\delta_{\psi_{s, \delta}}|_{s= 0}=\omega_{\varphi_{t,\delta}}^\delta.
\end{align}
Clearly,   the  twisted  Ricci potential  $h_{t,s, \delta}$ of  $\omega_{\psi_{s,\delta}}^\delta$ is given by
$$ h_{t, s, \delta}=- \frac{\partial}{\partial s} \psi_{s,\delta}+ const.$$
In particular,
$$ h_{t, \delta}=- \frac{\partial}{\partial s} \psi_{s,\delta}|_{s=0}+ const.$$

(\ref{twisted-K-R-flow}) reduces to a complex Monge-Amp\`ere flow,
\begin{align}\label{smooth-flow-MA-0}
& \frac{\partial}{\partial s} \tilde \psi =\log \frac{( \omega_{\varphi_{t, \delta}+\tilde\psi}^\delta)^n}{(\omega_{\varphi_{t, \delta}}^\delta)^n}+\mu\tilde  \psi  -  h_{t,\delta},\notag\\
& \tilde \psi_{0,\delta}=\tilde\psi_{\delta}(0,\cdot)= 0.
\end{align}
Here $h_{t,\delta}$ can be normalized so that $ h_{t, \delta}=-(\mu-t)\varphi_{t, \delta}$. Then $ h_{t, \delta}=- \frac{\partial}{\partial s} \tilde\psi_{s,\delta}|_{s=0}$
 Similarly  to  K\"ahler-Ricci flow  in [Ti97],  applying   the maximum principle to (\ref{smooth-flow-MA-0}), we have   following estimates (also see [LZ14]).
\begin{lem}\label{gradient-laplace-flow}

\begin{align} &1)~| \frac{\partial}{\partial s} \tilde\psi_{s,\delta}|^2+s|\nabla '\frac{\partial}{\partial s} \tilde\psi_{s,\delta}|^2\le e^{2\mu s} (\mu-t)^2 \|\varphi_{t,\delta}\|_{C^0(M)}^2.\notag\\
&2)  ~\Delta' (- \frac{\partial}{\partial s} \tilde\psi_{s,\delta})\ge e^{\mu s} \Delta h_{t,\delta}.\notag
\end{align}
Here $\Delta'$, $ \Delta$ are Laplace operators associated  to metrics $\omega_{\psi_{s, \delta}}^\delta, \omega_{\varphi_{t,\delta}}^\delta$, respectively.
  \end{lem}

\begin{lem}\label{holder-estimate-perturbation-flow}Let $v=v_{t, s,\delta}$  be a normalization of $h_{t,s,\delta}$ by  adding  a suitable constant  such  that
$$\int_M v(\omega_{\psi_{s,\delta }}^\delta)^n=0.$$
Let $\gamma=\frac{1}{8+8n}$.  Then there exists a
small number $\epsilon>0$ such that for any $t$ and $\varphi_{t,\delta}$
satisfying
\begin{align}\label{t-condition-phi}(\mu-t)^{1+
\frac{\gamma}{2}}\|\varphi_{t,\delta}\|_{C^0(M)}\le\epsilon,\end{align}
we have
 \begin{align}\label{holder-estimate-h-delta}\|v\|_{C^{\frac{1}{2}}(M)} \le
C\epsilon (\mu-t)^{\frac{\gamma}{2}}
\end{align}
  and
 \begin{align}\label{c0-perturbation}{\rm osc}_M \tilde\psi_{s,\delta}\le C\epsilon(\mu -t)^{\frac{\gamma}{2}}, ~\forall~ s\in [(\mu -t)^{2\gamma}, 1],
  \end{align}
    provided that  for any $s\in [(\mu -t)^{2\gamma}, 1]$ the first non-zero eigenvalue
    \begin{align}\label{large-eigenvalue}\lambda_1\ge \lambda_0>0
\end{align}
    of Laplace operator $\Delta'$  associated to the
metric $\omega_{\psi_{s,\delta}}^\delta$  and the following condition holds: there exists a
constant  $a>0$ such that for any $x_0\in M$ and $0<r<1$,
\begin{align}\label{volume-condition}{\rm vol} (B_r(x_0)) \ge ar^{2n}
\end{align}
 with respect to  $\omega_{\psi_{s,\delta}}^\delta$. Here $C=C(
a,\lambda_0)$ denotes a uniform constant depending only on the
constants $ a$ and $\lambda_0$.

\end{lem}

\begin{proof} Lemma \ref{holder-estimate-perturbation-flow} can be proved following the argument  of  smoothing lemma  in [Ti97] (also see Proposition 4.1 in [CTZ05]). In fact, under the  conditions (\ref{c0-perturbation})
and (\ref{large-eigenvalue}),  using the estimates  1) and 2)  in  Lemma \ref{twisted-K-R-flow},  we get
$$|v|
\le C( a,\lambda_0)
(1+\|h_{t,\delta}\|_{C^0(M)})(\mu -t)^{\frac
{3}{8(n+1)}},~\forall~
s\in[(\mu -t)^{2\gamma}, 1]. $$
 On the other hand,  again by the  estimate 1)  in  Lemma \ref{twisted-K-R-flow},  we have
 \begin{align}  \|\nabla' v\|^2_{C^0(M)}  \leq \frac{1}{s}
 e^2\|h_{t,\delta}\|^2_{C^0(M)},\notag
~\forall ~s>0.
\end{align}
 Combining  these two relations,   we derive
\begin{align}
\|v_s\|_{C^{\frac{1}{2}}(M)}
&\le C( a,
\lambda_0)(1+\|h_{t,\delta}\|_{C^0(M)})(\mu -t)^{\frac
{1}{8(n+1)}}\notag\\
&\le C [\epsilon+\|h_{t,\delta}\|_{C^0(M)} (\mu -t)^{\frac
{\gamma}{2}}] (\mu -t)^{\frac
{\gamma}{2}} ,~\forall~
s\in[(\mu -t)^{2\gamma}, 1]\notag.
\end{align}
 Note
  $$-(\mu-t)\varphi_{t,\delta}=h_{t,\delta}.$$
Thus under the assumption (\ref{t-condition-phi}),  we get (\ref{holder-estimate-h-delta}) immediately.

By the estimate 1) in Lemma \ref{gradient-laplace-flow},   we have
 \begin{align}| \frac{\partial
}{\partial s}
\tilde\psi|\le
e\|h_{t,\delta}\|_{C^0(M)},~\forall~s\le
1.\notag
\end{align}
 Note
 $$\tilde\psi=(\int_0^{(\mu-t)^{2\gamma}}+\int_{(\mu-t)^{2\gamma}}^s)(\frac{\partial }{\partial s}
 \tilde\psi).$$
Then by the assumption  (\ref{t-condition-phi}),  we obtain
\begin{align}\label{c0-perturbation-2} {\rm osc}_M \tilde\psi &\le (\mu-t)^{2\gamma} \sup_{s\in [0,1]} \|\frac{\partial }{\partial s}
 \tilde\psi \|_{C^0(M)}\notag\\
 & + \sup_{s\in [(\mu -t)^{2\gamma} ,1]} \|  \frac{\partial }{\partial s}
 \tilde\psi \|_{C^0(M)}\notag\\
 &\le C\epsilon (\mu -t)^{\frac{\gamma}{2}}.
\end{align}

\end{proof}

\section{Convergence of twisted K\"ahler-Ricci flows}

In this section, we  deal with  the local convergence of flows (\ref{twisted-K-R-flow}). First, similarly  to Lemma \ref{c2-estimate-twisted-metrics-lemma}, we have

\begin{lem}\label{c2-parabolic-twisted-KR} For any  $t\in (0,\mu),\delta\in(0,\delta_0]$,  it holds
\begin{align}\label{c2-parabolic-estimate-twisted-metrics}
C^{-1}  \kappa_\delta\le \omega^\delta_{\psi_{s,\delta}} \le C \kappa_\delta.
\end{align}
Here $C$ is a uniform  constant depends only on  metrics  $\hat\omega, \omega^*$ and norms of  $ \|\psi_{s, \delta}\|_{C^0(M)}$,  $\|\frac{\partial}{\partial s}\psi_{s,\delta}\|_{C^0(M)}.$
\end{lem}

\begin{proof}
Let $\bar\psi_\delta=\bar\psi_{s, \delta}=\psi_{s, \delta}+ \Psi_{\delta}-\Phi_{\delta}^\beta$. Then  by (\ref{smooth-flow-MA-0}), $\bar\psi=\bar\psi_\delta$ satisfies the following complex Monge-Amp\`ere flow,
\begin{align}\label{smooth-flow-MA}
& \frac{\partial}{\partial s} \bar \psi =\log \frac{(\kappa_\delta+\sqrt{-1}\partial\bar\partial\bar\psi)^n}{\kappa_\delta^n}+\mu\bar  \psi  - \bar h_{\kappa_\delta},\notag\\
& \bar \psi_{0,\delta}=\bar\psi_{\delta}(0,\cdot)= \varphi_{t,\delta}+\Psi_{\delta} -\Phi_{\delta}^\beta.
\end{align}

Following the estimate of (\ref{c2-yau}),  for the parabolic  equation  (\ref{smooth-flow-MA}),  we  get  Yau's $C^2$-estimate,
\begin{align} &(\frac{\partial}{\partial s}  -\Delta_{ \omega_{\psi_\delta}}^\delta) \log {\rm tr}_{\kappa_\delta} (\omega_{\psi_\delta}^\delta)\notag\\
&\le\frac{1}{{\rm tr}_{\kappa_\delta} (\omega_{\psi_\delta}^\delta)}\sum_{i<j} (\frac{1+\bar\psi_{ i\bar i}} {1+\bar\psi_{ j\bar j}} +\frac{1+\bar\psi_{ j\bar j}}{1+\bar\psi_{ i\bar i}} -2)R_{\delta i\bar ij\bar j} +\frac{1}{{\rm tr}_{\kappa_\delta} (\omega_{\psi_\delta}^\delta)} \Delta_{ \kappa_{\delta}}(\mu\bar  \psi_\delta - \bar h_{\kappa_\delta} )\notag
\end{align}
On the other hand, by  (\ref{laplace-h}), we have
\begin{align} &\frac{1}{{\rm tr}_{\kappa_\delta} (\omega_{\psi_\delta}^\delta)} \Delta_{ \kappa_{\delta}}(\mu\bar  \psi_\delta - \bar h_{\kappa_\delta} )\notag\\
&\le \frac{1}{{\rm tr}_{\kappa_\delta} (\omega_{\psi_\delta}^\delta)} \Delta_{ \kappa_{\delta}}  \log(\frac{\kappa_\delta^n}{\omega_0^n}\times (\delta+|S|^2_0)^{1-\beta}) +\frac{A}{{\rm tr}_{\kappa_\delta} (\omega_{\psi_\delta}^\delta)}
+\mu.
\notag
\end{align}
By  the Guenancia-Paun  inequality (\ref{gp-inequality}), it follows
\begin{align}&(\frac{\partial}{\partial s}  - \Delta_{ \omega_{\psi_\delta}^\delta}) (\log {\rm tr}_{\kappa_\delta} (\omega_{\psi_\delta}^\delta) +C_0 \Phi^\beta_\delta) \notag\\
&\le C_0' {\rm tr}_{\omega_{\psi_\delta}^\delta}(\kappa_\delta)+C_0', \notag
\end{align}
where $C_0$ and $C_0'$ are two uniform constants depending only on  metrics $\hat\omega$ and $\omega^*$. Hence by choosing a large number $B$, we deduce
 \begin{align}&(\frac{\partial}{\partial s}  - \Delta_{ \omega_{\psi_\delta}^\delta}) (\log {\rm tr}_{\kappa_\delta} (\omega_{\psi_\delta}^\delta) +C_0 \Phi^\beta_\delta -B\bar \psi_\delta ) \notag\\
&\le - {\rm tr}_{\omega_{\psi_\delta}^\delta}(\kappa_\delta)+C_0''. \notag
\end{align}
Now  we can apply the maximum principle to see that there exists a  uniform constant $C$, which depends only on $\hat\omega, \omega^*,   \|\psi_{s,\delta}\|_{C^0(M)}$ and  $\|\frac{\partial}{\partial s}\psi_{s,\delta}\|_{C^0(M)}$, such that
\begin{align}\omega_{\psi_{s,\delta}}^\delta=\kappa_\delta+ \sqrt{-1}\partial\bar\partial \bar\psi_{s,\delta} \ge C^{-1}\kappa_\delta.\notag
\end{align}
By (\ref{smooth-flow-MA}), we  also obtain
 \begin{align}\kappa_\delta  \ge  C'^{-1} \omega_{\psi_{s, \delta}}^\delta,\notag
\end{align}
where $C'=C'( \hat\omega, \omega^*, \|\psi_{s,\delta}\|_{C^0(M)}, \|\frac{\partial}{\partial s}\psi_{s,\delta}\|_{C^0(M)}).$
\end{proof}

\begin{theo}\label{perturbation-limit-from-flow}  For any  $s\in (0,1]$,  $\omega_{\psi_{s,\delta}}^\delta$ converge to a conic K\"ahler metric $\hat\omega_{\tilde\phi_{t,s}}=\hat\omega+\sqrt{-1}\partial\bar\partial \tilde \phi_{t,s}$
in sense of $C^{2,\alpha;\beta}$ K\"ahler potentials.

\end{theo}

\begin{proof}  By  the estimate 1) in  Lemma \ref{gradient-laplace-flow},  we have
\begin{align}\label{c0-psi-s} \|\psi_{s, \delta}\|_{C^0(M)}, \|\frac{\partial}{\partial s}\psi_{s,\delta}\|_{C^0(M)} \le e(\mu-t) \|\varphi_{t,\delta}\|_{C^0(M)}.
\end{align}
Then by   Lemma \ref {c2-parabolic-twisted-KR},  we see that there exists a  uniform constant  $C$, which depends only on $\hat\omega,\varphi_t$,  such that
$$C^{-1}  \kappa_\delta\le \omega_{\psi_{s, \delta}}^\delta \le C \kappa_\delta,~\forall ~\delta\in (0,\delta_0].$$
Thus the Sobolev constant  associated to  $\omega_{\psi_{s, \delta}^\delta}$  is uniformly bounded above  as same as  the metric $\kappa_\delta$ (cf. [LZ14]).
 Derivativing (\ref{smooth-flow-MA}) on $s$, we have
$$(\frac{\partial}{\partial s}  -\Delta_{ \omega_{\psi_{\delta}}^\delta }) \dot\psi_\delta=\mu\dot\psi_\delta,$$
where $\dot\psi_\delta=\dot\psi_{s,\delta}=\frac{\partial }{\partial s} \psi_{s,\delta}$.   Hence   the standard Moser iteration  method for the parabolic equation   implies
that there exist  a  positive number $\alpha$ and  a uniform constant $C$ such that
$$\sup_{x,y\in M} \frac{ |\dot \psi_{\delta}(x)-\dot\psi_{\delta}(y)|} {|x-y|_{\kappa_{\delta}}^\alpha}\le C.$$
As a consequece, $\dot \psi_{\delta}$ converges  to a  H\"older continueous function $f$ associated to the metric $\omega$ as $\delta\to 0$.  Namely,
$$\sup_{x,y\in M} \frac{ |f(x)-f(y)|} {|x-y|_{\omega}^\alpha}\le C.$$
  On the other hand, by the Kolodziej's H\"older estimate,  $ \psi_{\delta}$ are  uniformly  H\"older continuous functions, so they converge to
 a H\"older continuous function $\tilde\phi_t=\tilde\phi_{t,s}$ as $\delta\to 0$.  Moreover, $\tilde\phi_t$ is a current solution of following complex Monge-Amp\`ere equation,
\begin{equation}\label{limit-phi-delta-t-flow}
(\hat\omega+\sqrt{-1}\partial\bar\partial\phi)^n=e^{f+h_{\hat\omega}-\mu\phi}(\hat\omega)^n.
\end{equation}
By the reguarity theorem  in [JMR11],  it follows  that  $\tilde\phi_t$ is a $C^{2,\alpha';\beta}$-solution.  Hence $\hat\omega_{\tilde\phi_t}$ is a  conic K\"ahler metric.

\end{proof}

\begin{lem}\label{small-c0-puterbation-limit}For any $t $ and $\varphi_t$
satisfying
\begin{align}\label{t-condition}(\mu-t)^{1+
\frac{\gamma}{2}}\|\varphi_{t}\|_{C^0(M)}\le\epsilon,
\end{align}
 it holds
\begin{align}\label{c0-perturbation-2-limit} {\rm osc }_M (\tilde\phi_{t,s}-\varphi_t)\le C\epsilon(\mu -t)^{\frac{\gamma}{2}}, ~\forall ~s\in (0,1].
\end{align}
\end{lem}

\begin{proof} In  (\ref{c0-perturbation-2}),    we in fact prove
$$ {\rm osc }_M \tilde\psi_{s,\delta}\le C\epsilon(\mu -t)^{\frac{\gamma}{2}}, ~\forall ~s\in (0,1].$$
Then  (\ref{c0-perturbation-2-limit}) follows immediately from Theorem \ref{phi-delta-convergence} and Theorem \ref{perturbation-limit-from-flow} by taking $\delta\to 0$.
 \end{proof}

\begin{lem}\label{holder-estimate-perturbation}Let  $\tilde v=\tilde v_{t,s}$  be a normalization of $h_{\tilde\phi_{t,s}}$ by adding a  suitable constant such  that
$$\int_M  \tilde v (\hat\omega_{\tilde \phi_{t,s }})^n=0.$$
Then for any $t $ and $\varphi_t$
satisfying (\ref{t-condition}),
it holds
\begin{align}\label{l2-v}
 \|\tilde v_{t,s}\|_{C^{\frac{1}{2}}(M)} \le
C\epsilon (1-t)^{\frac{\gamma}{2}}, ~\forall ~s\in  [(\mu -t)^{2\gamma}, 1].
\end{align}

\end{lem}

\begin{proof} We claim: for the metric $\hat\omega_{ \tilde\phi_{t,s}}$, it holds
\begin{align}\label{metric-equivalence-t}
\frac{1}{2} \hat\omega\le \hat\omega_{\tilde\phi_{t,s}}\le 2\hat\omega.
\end{align}

By the above  claim together Theorem \ref{perturbation-limit-from-flow}, we see  that  there exists a small $\delta_0$ such that the conditions (\ref{large-eigenvalue}) and (\ref{volume-condition})  in Lemma \ref{holder-estimate-perturbation-flow}
are satisfied for metrics $\omega_{\psi_{s,\delta}}^\delta$ with $\delta\in(0,\delta_0]$. Note that  (\ref{t-condition-phi})  also holds by Theorem \ref{phi-delta-convergence}.
Then by    Lemma \ref{c2-parabolic-twisted-KR},   it follows that   (\ref{holder-estimate-h-delta})  holds for $v=v_{t,s,\delta}$ with   $\delta\in (0, \delta_0]$. Taking the limit of $v_{t,s,\delta}$ as $\delta\to 0$, we get  (\ref{l2-v}).

We prove   (\ref{metric-equivalence-t}) by contradiction. If  (\ref{metric-equivalence-t}) is not true, then there exists a $\psi\in\mathscr{H}^{2,\alpha;\beta}(M,\omega_0)$ such that
the solution $\varphi_t$ of  (\ref{backward_MAE})  on $t$ satisfies (\ref{t-condition}) and the  $C^{2,\alpha;\beta}$-norm of  K\"ahler potential $\tilde\phi_{t,s}$ in Theorem \ref{perturbation-limit-from-flow} satisfies
\begin{align}\label{big-estimate}\|\tilde\phi_{t,s}\|_{C^{2,\alpha;\beta}(M)}\ge A_0,
\end{align}
where $A_0$ is a positive  number.
On the other hand, from the proof of Theorem \ref{perturbation-limit-from-flow}, the  $C^{2,\alpha;\beta}$-norm of  $\tilde\phi_{t,s}$ depends on $\varphi_t$ continuously. Thus we may aslso assume that
$$\|\tilde\phi_{t,s}\|_{C^{2,\alpha;\beta}(M)}\le 2A_0$$
and
\begin{align}\label{metric-equivalence-t-2}
\frac{1}{4}  \hat\omega\le \hat\omega_{\tilde\phi_{t,s}}\le  4\hat\omega.
\end{align}

Once  (\ref{metric-equivalence-t-2}) holds,  we can use the above argument again to
 conclude that  there exists a small $\delta_0$ such that   (\ref{holder-estimate-h-delta})  holds for $v=v_{t,s,\delta}$ with   $\delta\in (0, \delta_0]$. Taking the limit of $v_{t,s,\delta}$ as $\delta\to0$, we get  (\ref{l2-v})
 for   $\hat\omega_{\tilde\phi_{t,s}}$.
Applying  Implicity Function  Theorem to  (\ref{limit-phi-delta-t-flow}), we  obtain
\begin{align}\|\tilde\phi_{t,s}\|_{C^{2,\alpha;\beta}(M)}\le C(\epsilon)\to 0,~{\rm as}~\epsilon\to 0.\notag
\end{align}
But this is impossible by (\ref{big-estimate}). The claim is proved.

\end{proof}

\section{Properness of  $F_{\omega_0,\mu}(\cdot)$}

By using the estimates at  last section,  we can improve Theorem \ref{Bern_thm} to

\begin{theo}\label{tian-zhu-conic-ke-metric-weak}    Suppose that  there exists  a conic   K\"ahler-Einstein metric $\omega=\omega_{CKE}$ on $M$  along
$D$ with cone angle $2\pi\beta\in (0,2\pi).$ Then there exists two uniform constants $\delta$  and $C$ such that
\begin{align}\label{nonlinear-ms}F_{\omega_0,\mu}(\psi)\ge \delta I(\psi)^{\frac{1}{8n+9}} -C,~\forall~\psi\in \mathscr{H}^{2,\alpha; \beta}(M,\omega_0).
\end{align}
\end{theo}

\begin{proof}

First,  by  the first relation in  (\ref{integral-dt}),  we get an identity
 $$F_{\omega,\mu}(\psi-\phi)=F_{\omega,\mu}(-\varphi_\mu)=\frac{1}{\mu}\int_0^\mu(I-J)_{\hat{\omega}}(\varphi_s)ds.$$
Then as in [TZ00], [CTZ05],  we obtain
\begin{align}\label{f-inequiality}&F_{\omega,\mu}(\psi-\phi)\notag\\
&\ge \frac{1}{\mu}(\mu-t)(I-J)_{\hat\omega}(\varphi_t)\notag\\
&\ge \frac{1}{n\mu}(\mu-t)J_{\hat\omega}(\varphi_t)\notag\\
&\ge  \frac{1}{n\mu}(\mu-t)J_{\hat\omega}(\varphi_\mu)   -  \frac{1}{n\mu}(\mu-t) {\rm osc}_M(\varphi_t-\varphi_\mu)\notag\\
&\ge  \frac{1}{n\mu(n+1)}(\mu-t)I_{\hat\omega}(\varphi_\mu)   -  \frac{1}{n\mu}(\mu-t) {\rm osc}_M(\varphi_t-\varphi_\mu).
\end{align}

Next,   for a small $\epsilon$,  we choose  a $t$
such that
\begin{align}\label{special-condition} (\mu-t)^{1+\frac{\gamma}{2}}\|\varphi_t\|_{C^0(M)}=\epsilon.
\end{align}
 Without loss of generality, we may assume that the above inequality can be obtained,
  otherwise $\|\varphi_t\|_{C^0(M)}$ is  unifromly bonuded and  the situation will be simple.
   Then by   Theorem \ref{perturbation-limit-from-flow} and  Lemma \ref{small-c0-puterbation-limit},   there exists  a  $C^{2,\alpha'; \beta}$ K\"ahler potential  $\tilde\varphi_t$ such that
\begin{align}{\rm osc}_M(\varphi_t-\varphi_\mu)&  \le {\rm osc}_M (\tilde\varphi_t -\varphi_\mu)+ {\rm osc}_M( \tilde \varphi_t -\varphi_t)\notag\\
&\le {\rm osc}_M (\tilde\varphi_t -\varphi_\mu)+  C\epsilon(\mu -t)^{\frac{\gamma}{2}}.
\notag
\end{align}
On the other hand,   by   Lemma \ref{holder-estimate-perturbation},  we can apply  Implicity Function  Theorem to  (\ref{limit-phi-delta-t-flow}) to get
 $${\rm osc}_M (\tilde\varphi_t -\varphi_\mu)\le C(\epsilon)\to 0,~{\rm as}~ \epsilon\to 0.$$
 Thus
\begin{align}\label{osc-phi-t-phi-mu}{\rm osc}_M(\varphi_t-\varphi_\mu)\le C(\epsilon).
\end{align}

Combining (\ref{f-inequiality}) and (\ref{osc-phi-t-phi-mu}), we have
 \begin{align}\label{f-proper-1}F_{\omega,\mu}(\psi-\phi)\ge  \frac{1}{n\mu(n+1)}
(\mu-t)I(\varphi_\mu)-C.
\end{align}
 Note
 $$\|\varphi_t\|_{C^0(M)}\le {\rm osc}_M(\varphi_t).$$
Then by (\ref{osc-phi-t-phi-mu}), we get
  \begin{align}\label{osc-phi-t}\|\varphi_t\|_{C^0(M)}\le {\rm osc}_M(\varphi_\mu)+1.
\end{align}

In a special case, we assume that the K\"ahler potential $\psi$ satisfies
\begin{align}\label{osc-condition}
{\rm osc}_M\psi\le I_{\omega_0}(\psi)+C_0,
\end{align}
where $C_0$ is a uniform constant.
Then by the relation (\ref{special-condition}) and (\ref{osc-phi-t}), a simple computation shows
\begin{align}\label{special-case-1} F_{\omega,\mu}(\psi-\phi)&\ge \delta
I_{\omega_0}(\varphi_\mu)^{\frac{1}{8n+9}}-C'\notag\\
&\ge \delta
I_{\omega_0}(\psi)^{\frac{1}{8n+9}}-C',
\end{align}
 where $\delta,  C'>0$ are two uniform constants which depending only on the choice of $\epsilon$ in (\ref{special-condition}).   Using  the cocycle condition in  (\ref{cocycle-condition}), we derive immediately,
\begin{align}\label{special-case}F_{\omega_0,\mu}(\psi)\ge \delta
I_{\omega_0}(\psi)^{\frac{1}{8n+9}}-C''.
\end{align}
In  general case, we can use a trick in [TZ00] to derive (\ref{special-case}) for $\psi$. In fact, we can first  apply  (\ref{special-case})  for solutions $\varphi_t$ with $t\ge\epsilon_0>0$ to get an estimate for
${\rm osc}_M(\varphi_t-\varphi_\mu)$,  then by  the relation in  (\ref{f-inequiality})
 we   obtain  (\ref{special-case}) for $\psi$.

\end{proof}

\begin{proof}[End of proof of  Theorem  \ref{tian-zhu-conic-ke-metric}]Theorem  \ref{tian-zhu-conic-ke-metric} is an improvement of Theorem \ref{tian-zhu-conic-ke-metric-weak}.  By Lemma \ref{functional-smoothing},
 we  suffice  to obtain  the esitimate (\ref{proper-inequality}) in Theorem  \ref{tian-zhu-conic-ke-metric} for
K\"ahler potentials in $\mathscr{H}^{2,\alpha; \beta}(M,\omega_0)$.
 It  was observed by Phong-Song-Strum-Weinkove that (\ref{proper-inequality})  can  come from  (\ref{nonlinear-ms})
in case of K\"ahler-Einstein metric [PSSW08].   In fact, as in [TZ00], by   (\ref{nonlinear-ms}) for  solutions $\varphi_t$ with $t\ge\epsilon_0>0$, they further show that there exists a $t_0$ with $\mu-t_0\ge \delta_0>0$ (where $\mu=1$)
  for some  uniform constant $\delta_0$  such that
$${\rm osc}_M(\varphi_{t_0}-\varphi_\mu)\le A,$$
where $A$ is  a uniform constant which depends only on  the K\"ahler-Einstein metric.  We show  that such choice of $t_0$ can be done similarly  in our case of  conic K\"ahler-Einstein metric $\omega_\phi$ as follows.

By  the first  relation in (\ref{integral-dt}) together with  the equation (\ref{backward_MAE}), we have
\begin{align} &F_{\omega,\mu}(\varphi_t-\varphi_\mu)=F_{\hat\omega,\mu}(\varphi_t)-F_{\hat\omega,\mu}(\varphi_\mu)\notag\\
&= -\frac{1}{t}\int_0^t (I-J)_{\hat{\omega}}(\varphi_s)ds+\frac{1}{\mu}\int_0^{\mu} (I-J)_{\hat{\omega}}(\varphi_s)ds -\frac{1}{\mu} \log(\frac{1}{V}\int_M e^{ (t-\mu)\varphi_t}\hat\omega_{\varphi_t}^n)\notag\\
&\le-\frac{1}{t}\int_0^t (I-J)_{\hat{\omega}}(\varphi_s)ds+\frac{1}{\mu}\int_0^{\mu} (I-J)_{\hat{\omega}}(\varphi_s)ds+ \frac{\mu-t}{\mu V}\int_M  \varphi_t\hat\omega_{\varphi_t}^n.\notag
\end{align}
Note that  the first  relation in (\ref{integral-dt}) is equivalent to
$$- \frac{1}{V} \int_M \varphi_t\hat\omega_{\varphi_t}^n= (I-J)_{\hat{\omega}}(\varphi_t)-\frac{1}{t} \int_0^t  (I-J)_{\hat{\omega}}(\varphi_s)ds.$$
It follows
\begin{align}\label{f-upper-bound} F_{\omega,\mu}(\varphi_t-\varphi_\mu)&\le  \frac{1}{\mu} \int_t^{\mu}  (I-J)_{\hat{\omega}}(\varphi_s)ds -\frac{\mu-t}{\mu} (I-J)_{\hat{\omega}}(\varphi_t) \notag\\
 &\le \frac{\mu-t}{\mu} [(I-J)_{\hat{\omega}}(\varphi_\mu) - (I-J)_{\hat{\omega}}(\varphi_t)]\notag\\
&\le \frac{n(\mu-t)}{\mu}{\rm osc}_M(\varphi_t-\varphi_\mu).
\end{align}

On the other hand,  by the Green formula in [JMR11],  (\ref{osc-condition}) holds for $\varphi_t-\varphi_\mu$ whenever $t\ge \epsilon_0>0$
 since  Ricci curvature of $\hat\omega_{\varphi_t}=\omega+\sqrt{-1}\partial\bar\partial (\varphi_t-\varphi_\mu)$  is strictly positive.
Then applying  (\ref{special-case-1}) for $\varphi_t-\varphi_\mu$, we see that there exist two constants $A_0, C>0$  such that
$$F_{\omega,\mu}(\varphi_t-\varphi_\mu)\ge  A_0
I_\omega(  \varphi_t-\varphi_\mu)^{\frac{1}{8n+9}}-C, ~\forall ~t\ge\epsilon_0.$$
Combining this with (\ref{f-upper-bound}), we derive
\begin{align}\label{i-functional-relation}
I_{\omega} ( \varphi_t-\varphi_\mu)^{\frac{1}{8n+9}} [A_0-\frac{n(\mu-t)}{\mu} I_{\omega} (  \varphi_t-\varphi_\mu)^{1-\frac{1}{8n+9}}]\le C'.
\end{align}

Case 1:  For any $t\in [\frac{\mu}{2},\mu]$,  it holds
 $$(\mu-t) I_{\omega} (  \varphi_{t}-\varphi_\mu)^{1-\frac{1}{8n+9}}<\frac{A_0\mu}{2n}.$$
Then   we can choose $t=t_0=\frac{\mu}{2}$ so that
$ I_{\omega} (  \varphi_{t_0}-\varphi_\mu)$,  and also ${\rm osc}_M( \varphi_{t_0}-\varphi_\mu)$ is uniformly bounded.
 Thus by  the relation in (\ref{f-inequiality}), we get
 (\ref{proper-inequality}).  The proof is finished.

Case 2:  There exists a $t_0\in [\frac{\mu}{2},\mu]$ such that
  \begin{align}\label{t-condition-2}  (\mu-t_0) I_{\omega} (  \varphi_{t_0}-\varphi_\mu)^{1-\frac{1}{8n+9}}=\frac{A_0\mu}{2n}.
\end{align}
By the  above choice of $t_0$, from (\ref{i-functional-relation})  it is easy to see that   ${\rm osc}_M( \varphi_{t_0}-\varphi_\mu)$ is uniformly bounded.
Again by  (\ref{t-condition-2}), we get $\mu-t_0\ge \delta_0>0$ for some  uniform constant $\delta_0$. The theorem is proved.

There is another way to get (\ref{proper-inequality}) by using the Donaldson's openness   theorem,  Theorem \ref{Do_openness} in next section. This is observed in [LS12].

\end{proof}

\section{A new proof of Donaldson's openness   theorem}

In this section, we apply Theorem \ref{tian-zhu-conic-ke-metric-weak}  to  prove  the following Donaldson's openness   theorem.

\begin{theo}\label{Do_openness}Let $D$ be a smooth divisor of a Fano manifold $M$ with $[D]\in \lambda c_1(M)$ for some $\lambda>0$ such  that there is no non-zero holomorphic field
which is tangent to $D$ along $D$. 
Suppose that there exists  a conic  K\"ahler-Einstein metric on $M$ with  cone angle $2\pi\beta_0\in (0,2\pi)$ along $D$. 
 Then for any $\beta$ close to $\beta_0$ there exists a conic K\"ahler-Einstein metric  with  cone angle $2\pi\beta$.
\end{theo}

\begin{proof}  Let $\mu_0= 1-\lambda(1-\beta_0)$.   Then  $F_{\omega_0,\mu_0}(\cdot)$ is proper by Theorem \ref{tian-zhu-conic-ke-metric-weak}. Thus
   twisted $F$-functionals $F_{\mu_0,\delta}(\cdot)$ defined by (\ref{twisted-f-functional-t})  are all proper for
any $\delta\in (0,\delta_0]$.  By the argument in Section 3,   it follows  that  there exists a solution of  (\ref{smooth-continuity-modified})  on $t=\mu_0$ for any  $\delta\in (0,\delta_0]$.  Hence,  for a fixed $\delta=\frac{\delta_0}{2}$, we
apply  Implicit  Function  Theorem to see that there exists an $\epsilon_0$ such that (\ref{smooth-continuity-modified})   is  solvable for any  $\mu\in [\mu_0,\mu_0+\epsilon_0)$. Note that  the  twisted Ricci potential
  $h_\delta$ of $\omega_\delta$
in  (\ref{smooth-continuity-modified}) satisfies
(\ref{twisted-h-delta}) with $\beta=1- \frac{1-\mu}{\lambda}$.
This means that there exists a twisted K\"ahler-Einstein metric
associated to positive $(1,1)$-form $\Omega=  (1-\beta)\eta_{\frac{\delta_0}{2}}$ for any $\mu\in [\mu_0,\mu_0+\epsilon_0)$.   By a result of  X. Zhang and X. W. Zhang [ZhZ13],  the twisted $F$-functionals $F_{\mu,\frac{\delta_0}{2}}(\cdot)$ is  proper for any $\mu\in [\mu_0,\mu_0+\epsilon_0)$. In fact,  they prove a version of Theorem \ref{tian-zhu-conic-ke-metric-weak} on a  Fano manifold  which admits a twisted K\"ahler-Einstein metric.

 By  a direct computation,  it is easy to see that for any $\delta\in (0,\delta_0)$ it holds
$$|F_{\mu,\frac{\delta_0}{2}}(\psi)-F_{\mu,\delta}(\psi)|\le C(\|\Psi_{\delta}\|_{C^0(M)}, \|\Psi_{\frac{\delta_0}{2}}\|_{C^0(M)})\le C, ~\forall~\psi\in  \mathcal H(M,\omega_0).$$
This implies that $ F_{\mu,\delta}(\cdot)$ are all proper for any  $\delta\in (0,\delta_0)$ and $\mu\in [\mu_0,\mu_0+\epsilon_0)$.
Thus by the argument in Section 3,  we see that  there exists a solution $\varphi_{\mu, \delta}$
 of  (\ref{smooth-continuity-modified})  on $t=\mu\in  [\mu_0,\mu_0+\epsilon_0)$ for any  $\delta\in (0,\delta_0)$.
Moreover,
$${\rm osc}_M\varphi_{\mu, \delta}\le C,$$
  where  $C$   is a uniform constant   independent of $\mu$ and $\delta$.
On the other hand,  the estimate (\ref{c2-estimate-twisted-metrics})  in Lemma \ref{c2-estimate-twisted-metrics-lemma}  also holds for metrics $\omega_{\varphi_{\mu, \delta}}^\delta$.
 By taking a sequence $\delta_i\to 0$,   $\varphi_{\mu,\delta_i}$ converge   to a H\"older continueous function $\varphi_\infty -\psi$  which satisfies the   weak conic K\"ahler-Einstein metric equation,
$${\rm Ric}(\omega_{\varphi_\infty})=\mu\omega_{\varphi_\infty}+2\pi  (1-\beta)[D], ~\mu\in [\mu_0, \mu_0+\epsilon_0)$$
with property:
$$C^{-1}  \hat\omega\le \omega_{\varphi_\infty} \le C \hat\omega, ~{\rm in}~ M\setminus D$$
for some uniform positive number  $C$.
By  the regularity theorem in [JMR11],   $\omega_{\varphi_\infty}$ is a conic K\"ahler-Einstein metric in sense of $C^{2,\alpha;\beta}$ K\"ahler potentials.
The proof of  Theorem \ref{Do_openness} is completed.

\end{proof}

\end{document}